\documentclass[12pt]{amsart}
\usepackage{amsmath}
\usepackage{amsxtra}
\usepackage{amstext}
\usepackage{amssymb}
\usepackage{amsthm}
\usepackage{latexsym}
\usepackage{dsfont} 
\usepackage{verbatim}
\usepackage{tabls}
\usepackage{graphicx}
\usepackage{rotating}
\usepackage{hyperref}
\usepackage{caption}
\usepackage{xcolor}

\usepackage{cite}

\theoremstyle{plain}
\newtheorem{thm}{Theorem}

\newtheorem{prop}[thm]{Proposition}
\newtheorem{ques}[thm]{Question}

\theoremstyle{definition}

\newtheorem{rem}[thm]{Remark}

\numberwithin{equation}{section}

\def\dist{\operatorname{dist}}
\def\supp{\operatorname{supp}}

\def\eps{\varepsilon}
\def\kap{\varkappa}
\def\R{\mathbb{R}}

\def\wh{\widehat}
\def\wt{\widetilde}
\def\spec{\operatorname{spec}}

\def\S{\mathbb{S}}

\def\kap{\varkappa}

\def\Z{\mathbb{Z}}

\def\N{\mathbb{N}}

\def\N{\mathbb{N}}
\def\U{\mathcal{U}}

\def\XXint#1#2#3{{\setbox0=\hbox{$#1{#2#3}{\int}$}
     \vcenter{\hbox{$#2#3$}}\kern-.5\wd0}}

\newcommand{\ds}{(-\Delta+1)^{s/2}}
\newcommand{\dss}{(-\Delta+1)^{s/4}}

\newcommand{\dssi}{(-\Delta+1)^{-s/4}}
\newcommand{\cala}{\mathcal{A}}

\newcommand{\lip}{\langle}
\newcommand{\rip}{\rangle}
\newcommand{\hls}{{H^{s/2}\times L^2}}
\DeclareMathOperator{\re}{Re}

\title[Geometric control]{Uncertainty principles associated to sets satisfying the Geometric Control Condition}
\author{Walton Green}

\address{School of Mathematical and Statistical Sciences, Clemson University, Clemson, SC, USA}
\email{awgreen@clemson.edu}

\author{Benjamin Jaye}
\email{bjaye@clemson.edu} 

\thanks{B.J. supported by NSF through DMS-1847301 and DMS-1800015.}

\author{Mishko Mitkovski}

\email{mmitkov@clemson.edu}

\thanks{M.M. supported by NSF through DMS-1600874.}

\begin{document}
\maketitle

\begin{abstract}
In this paper, we study forms of the uncertainty principle suggested by problems in control theory. We obtain a version of the classical Paneah-Logvinenko-Sereda theorem for the annulus. More precisely, we show that a function with spectrum in an annulus of a given thickness can be bounded, in $L^2$-norm, from above by its restriction to a neighborhood of a  GCC set, with constant independent of the radius of the annulus. We apply this result to obtain energy decay rates for damped fractional wave equations, extending the work of Malhi and Stanislavova to both the higher-dimensional and non-periodic setting.
\end{abstract}

\section{Introduction}

The goal of this paper is to investigate versions of the uncertainty principle suggested by control theory for PDEs. There is a long history of the relationship between these two fields beginning with Riesz sequence problems for non-harmonic Fourier series and their application to both wave and heat equations by the so-called ``moment method'' of D.~L. Russell \cite{russell67,avdonin95,jaffard90}. More recently, inequalities of the uncertainty principle type have found application to control theory on unbounded domains and compact Riemannian manifolds of negative curvature \cite{egidi18,lebeau19,green20decay,bourgain18,dyatlov2016spectral,dyatlov2018semiclassical,jin18,jaye18,han18}

Fix $k\in \{1, \dots, d\}$.  For $\ell>0$ and $\gamma>0$, a set $E\subset \R^d$ satisfies the $k$-dimensional $(\ell,\gamma)$-\emph{geometric control condition} (GCC) if for any $k$-dimensional cube $Q\subset \R^d$ of side-length at least $\ell$,
$$\mathcal{H}^k(Q\cap E)\geq \gamma \ell^k,
$$
where $\mathcal{H}^k$ denotes the $k$-dimensional Hausdorff measure (which is just the $k$-dimensional Lebesgue measure on the $k$-dimensional plane containing $Q$). We say that $E$ satisfies the $k$-GCC if it satisfies the $k$-dimensional $(\ell,\gamma)$-GCC for some $\ell>0$ and $\gamma$.\\

Fubini's theorem ensures that if $E$ satisfies the $k$-GCC, then $E$ satisfies the $\ell$-GCC for $\ell\geq k$.  The two extreme cases, $k=1$ and $k=d$, have a substantial literature:\\

The 1-GCC arises in the study of control theory for hyperbolic equations in the work of Bardos, Lebeau, Rauch, Taylor and Phillips \cite{b-l-r,rauch74}. Given a Laplacian, an open set $\omega$ satisfies the GCC if for some $T>0$, the Hamiltonian flow always intersects $[0,T] \times \omega$. Choosing the Laplacian to be $-\Delta$ on $\R^d$, and removing the regularity condition on $\omega$, this simplifies to the above condition with $k=1$.\\

On the other hand, when $k=d$, we recover the definition of relatively dense, or thick, sets which are characterized by the Paneah-Logvinenko-Sereda (PLS) theorem \cite{paneah61,logvinenko74,kovrijkine01} as sets $E$ for which 
\begin{equation}\label{eq1}\|f\|_{L^p(\R^d)} \le C \|f\|_{L^p(E)}\end{equation}
for all $f$ satisfying $\spec (f) \subset B$. Here $\spec(f)$ denotes the support of the Fourier transform of $f$ and $B$ is a $d$-dimensional ball of fixed radius. The precise quantitative dependence of the above constant $C$ on the parameters of $E$ and the radius of $B$ is given by Kovrijkine in~\cite{kovrijkine01}, where it was shown that $C$ depends exponentially on the radius of the ball $B$. \textcolor{black}{Developing variants of the PLS theorem in different settings is an active area of research \cite{hartmann2020quantitative,egidi2020scale,hartmann2021dominating,ghobber13,ortega2013carleson}}\\

Our main goal is to extend the PLS theorem to functions with Fourier support in a spherical shell of a fixed width. The main point is that the constant in the corresponding inequality is now dependent only on the width of the spherical shell and not on its radius. For a set $A\subset \R^d$, and $\delta>0$, $\U_{\delta}(A)$ denotes the open $\delta$-neighborhood of $A$ in $\R^d$.


\begin{thm}\label{sphere}
Let $E$ satisfy the $1$-GCC. For any $\beta,\delta>0$, there exists $C>0$ such that
	\begin{equation}\label{eq:ann} \|f\|_{L^2(\R^d)} \le C\|f\|_{L^2(\U_\delta(E))} \end{equation}
whenever $f \in L^2(\R^d)$ satisfies $\spec f \subset \{ \xi \in \R^d: R-\beta \le |\xi| \le R+\beta \}$ for some $R>0$.
\end{thm}

To reiterate, the novelty of this result is that \emph{the inequality (\ref{eq:ann}) holds with a constant $C$ independent of $R$}. If we don't care about the form of the constant $C$ this inequality is, of course, a direct consequence of the PLS Theorem (by placing the annulus $\{\bigr| |\xi|-R\bigr| \le \beta \}=\U_\beta(R\S^{d-1})$ inside a ball of radius $R+\beta$, but doing so yields a bound in (\ref{eq:ann}) which grows exponentially with $R$).  Our method proves an explicit bound for $C$ in terms of the $\beta$ and $\delta$, and the GCC parameters $\ell$ and $\gamma$.

Since the sphere is a $(d-1)$-dimensional manifold, Theorem \ref{sphere} does not hold if we replace a $1$-GCC set $E$ with a $k$-GCC set for any $k> 1$.  The following theorem is our main result, which shows that, for $k\in \{1,\dots, d-1\}$, if the set $E$ is $k$-GCC, then the spherical shell in Theorem~\ref{sphere} can be replaced by a neighborhood of any compact, $C^{1}$-smooth, $(d-k)$-dimensional submanifold. Again the main point is that the constant $C$ only depends on the width of the neighborhood $\beta$, and not on $R$ (see below).

\begin{thm}\label{manifold}  
Fix $k\in \{1,\dots, d-1\}$.  Suppose that $\Sigma$ is a compact, $C^{1}$-smooth, $(d-k)$-dimensional submanifold of $\R^d$.  Suppose that $E$ satisfies the $k$-GCC. Then for any $\delta>0$ and $\beta >0$, there is a constant $C>0$ such that for every $R>0$,
$$\|f\|_{L^2(\R^d)}\leq C\|f\|_{L^2(\U_{\delta}(E))}$$
whenever  $f\in L^2(\R^d)$ satisfies $\spec(f)\subset \U_{\beta}(R\Sigma)$.

\end{thm}

There is a history of developing sampling inequalities for band limited functions from lower-dimensional sets\cite{strichartz1989uncertainty,jaming2016uncertainty}.
Theorem \ref{manifold} continues to hold in the case $k=d$ if $\Sigma$ is defined to be a finite set of points.  This can be derived by inspection from our proof, but is in fact a result of Kovrijkine~\cite{kovrijkine01} which holds for all relatively dense (or $d$-GCC) sets $E$ and not only for neighborhoods $\U_{\delta}(E)$ of such sets. We wonder if this is the case for other co-dimension $k$, and we specifically pose the following question.

\begin{ques}\label{annques} Suppose that $E$ satisfies the $1$-GCC.  Does there exist $C>0$ such that for every $R>0$,
$$\|f\|_{L^2(\R^d)}\leq C\|f\|_{L^2(E)}
$$
for every $f\in L^2(\R^d)$ with $\spec(f)\subset \U_1(R\mathbb{S}^{d-1})$?
\end{ques}

\subsection{Application to the decay of damped wave equations}

Our results above were inspired by the following energy decay rate problem for damped fractional wave equation. 

Fix $s>0$ and a damping function $\gamma : \R^d \to \R_{\ge 0}$. Consider the fractional damped wave equation recently introduced by Malhi and Stanislavova in \cite{malhi19}.\\

For $(x,t) \in \R^d \times \R_{\ge 0}$, let $w$ satisfy
 	\begin{equation}\label{eq:1} w_{tt}(x,t) +\gamma(x)w_t(x,t) + \ds w(x,t)=0.\end{equation}
The damping force is represented by $\gamma w_t$ and the fractional Laplacian is defined, for $r \in \R$ by
	\[ (-\Delta+1)^{r}f(x) = \int_{\R^d} (|\xi|^2+1)^r \wh f(\xi) e^{ix \xi} \, d\xi. \] 
Herein, we study the decay rate of the energy of $w$, defined by
	\begin{align*} E(t) &= \|(w(t),w_t(t))\|_{\hls} \\
			&= \left(\int_{\R^d} |\dss w(x,t)|^2 + |w_t(x,t)|^2 \, dx \right)^{1/2}.\end{align*}
Standard analysis shows that if $\gamma=0$, then the energy is conserved, i.e. there is no decay. On the other hand, for constant damping $\gamma=c >0$, it can be shown that $E(t)$ decays exponentially.\\

The classical case of $s=2$ has been well-studied on bounded domains in the pioneering works of Bardos, Lebeau, Rauch, Taylor, and Phillips \cite{b-l-r,rauch74}. Recently, Burq and Joly have extended these results to $\R^d$ \cite{burq16}, and in particular showed that if $\gamma$ is uniformly continuous and satisfies the GCC condition (\ref{eq:gamma-gcc}) below, then $E(t)$ decays exponentially in $t$.  The methods in these works are that of microlocal and semiclassical analysis for which we refer to the book of Zworski \cite{zworski-book}. These techniques, which allow one to handle very general Laplacians, impose regularity constraints on the damping coefficient $\gamma$. \\

We note two recent works which have, in one dimension, utilized Fourier analysis to prove exponential decay for rough damping \cite{green20decay,malhi18}.  Fourier analytic methods have also proved useful in understanding (polynomial, or logarithmic) decay rates of the semi-group under weaker conditions than the GCC \cite{anantharaman2014sharp, wunsch17}.\\


Using Theorem~\ref{sphere}, we will prove a resolvent estimate (Proposition~\ref{prop:res} below) for the fractional Laplacian, which then, using the strategy in~\cite{green20decay}, yields a new proof of the Burq-Joly theorem~\cite{burq16} \textcolor{black}{as a special case, and extends the results of \cite{malhi19} to higher dimensions and to non-periodic damping.}

\begin{thm}\label{thm:wave}
Suppose $\gamma$ is a non-negative, bounded, uniformly continuous function. There exists $L>0$ and $c>0$ such that
	\begin{equation}\label{eq:gamma-gcc} \int_\ell \gamma(x) dm_1(x) \ge c >0 \end{equation}
for all line segments $\ell \subset \R^d$ of length $L$ if and only if for every $s>0$ there exists $C,\omega>0$ such that
	\[ E(t) \le \left\{ \begin{array}{lrl} C(1+t)^{\tfrac{-s}{4-2s}}\|w(0),w_t(0)\|_{H^s \times H^{s/2}} & \mbox{if}& 0 < s< 2\\[2mm] Ce^{-\omega t}E(0) & \mbox{if}& s \ge 2 \end{array} \right. \]
for all $t>0$.	
\end{thm}

We reiterate that, in the case $s=2$, Theorem \ref{thm:wave} was proved by Burq and Joly \cite{burq16} using semiclassical analysis.  Our main goal here was to show how such results follow directly from uncertainty principles. This yields a particularly elementary proof. 
The compactness methods of \cite{burq16} enable one to prove Theorem \ref{sphere} for $\beta$ small enough depending on $E$ and $\delta$, but we do not know how to obtain the full strength of Theorem~\ref{sphere} by these methods.  Consequently, the uncertainty principles developed here may have other applications to control theory problems for wave equations.
Burq and Joly also pose the question of whether the result can be proved without the assumption of uniform continuity.  This would follow from a positive answer to Question~\ref{annques}.

\section{A Paneah-Logvinenko-Sereda theorem for strips}

The following result, which we believe is of interest by itself, is an important ingredient in the proof of our main theorems. It can be viewed as a PLS theorem for strips, in a sense that it characterizes the sets which (up to modification by sets of $m_d$-measure zero) satisfy the $k$-GCC, $k \in \{1,\ldots,d\}$, as those for which (\ref{eq1}) holds whenever $\spec(f)$ is contained in a $(d-k)$-dimensional ``strip.''  To precisely state the result, we need to introduce some additional notation.\\

A $(d-k)$-plane is a $(d-k)$-dimensional affine plane (which we interpret as a single point if $k=d$). For a set $S\subset \R^d$, we define\footnote{The notation comes from Peter Jones' analysts traveling salesman problem \cite{jones1990rectifiable}.} 
$$\beta_{d-k}(S) = \inf_{\substack{L\text{ is a }\\d-k\text{ plane}}}\sup_{x\in S}\dist(x,L)
$$
Therefore, if $\beta_{d-k}(S)< \beta$ then there is a $(d-k)$-plane $L_S$ such that $S\subset \U_{\beta}(L_S)$.

\begin{thm}\label{GCCcharac}
Fix $k\in \{1,\dots, d\}$, $\ell>0$ and $\gamma>0$, $p\in [1,\infty)$.  For every $\beta>0$ there exists $C>0$ such that if $E$ satisfies the $k$-dimensional $(\ell,\gamma)$-GCC, and $f\in L^p(\R^d)$ satisfies $\beta_{d-k}(\spec(f))<\beta$, then
    $$\|f\|_{L^p(\R^d)}\leq C\|f\|_{L^p(E)}.
    $$
\end{thm}


The constant $C$ in Theorem \ref{GCCcharac} will take the form
$$C= \Bigl(\frac{C_0 }{\gamma}\Bigl)^{C_0\beta \ell},
$$
where $C_0=C_0(k)>0$.  The proof of Theorem \ref{GCCcharac} is a modification of the proof of the aforementioned PLS theorem given by Kovrijkine ~\cite{kovrijkine01}, where it is shown (in the case $k=d$) that the form of constant we obtain is sharp (up to the value of $C_0$).\\

Provided that one handles sets of measure zero appropriately, the $k$-GCC condition is necessary for the conclusion to hold, see Proposition \ref{PLSprop} below.\\


Without loss of generality, we may assume that a best approximating plane for $\beta_{d-k}(\spec(f))$ is the plane $\R^{d-k}$.   Theorem \ref{GCCcharac} therefore follows from the following more precise proposition. In this statement, and throughout the paper, for $k\in \N$, $m_k$ denotes the $k$-dimensional Lebesgue measure on $\R^k$.

\begin{prop}\label{PLSprop}  Fix $E\subset \R^d$.  The following two conditions are equivalent:

\begin{enumerate}
\item There exist $\ell>0$ and $\gamma\in (0,1)$ such that for $m_{d-k}$-almost every $x'\in \R^{d-k}$, whenever $Q\subset \R^k$ is an cube of length at least $\ell$, then
\begin{equation}\label{coordinatecubes}m_k(\{t\in Q: (t,x')\in E\})\geq \gamma\ell^k.\end{equation}
\item For every $\beta>0$, there exists $C>0$ such that if $f\in L^p(\R^d)\cap L^2(\R^d)$ satisfies $\supp(\wh{f}\,)\subset [-\beta,\beta]^k\times \R^{d-k}$, then
$$\int_{\R^d}|f|^pdm_d\leq C\int_E |f|^pdm_d.
$$
\end{enumerate}
Moreover, in the direction (1)$\implies$(2) we will prove that the constant $C$ takes the form
$$C=\Bigl(\frac{C_2}{\gamma}\Bigl)^{C_2\beta\ell}
$$ 
for a constant $C_2$ depending on $k$.
\end{prop}

\begin{proof}  We first shall prove \textbf{(2)$\implies$(1)}.  The proof in this direction follows an idea of Paneah \cite{paneah61} as presented Havin and Joricke's book \cite{havin12}.  Fix $p\in (1,\infty)$.  By a simple covering argument, if the inequality (\ref{coordinatecubes}) for all cubes of length $\ell$, then it holds (with $\gamma$ replaces by $c\gamma$ for an absolute constant $c$) for every $\ell'\geq \ell$.  Therefore if suffices to find some $\ell>0$ such that (\ref{coordinatecubes}) for cubes of length $\ell$.

Suppose that $f(t,x')=g(t)[h(x')]^{1/p}$, where $g\in L^p(\R^{k})$ satisfies $\spec(g)\in [-\beta, \beta]^k$, and $h\in L^1(\R^{d-k})$, $h\geq 0$ and $\|h\|_1=1$.  Then, by Tonelli's theorem,
$$1\leq C\int_{\R^{d-k}}h(x')\int_{\{t\in \R^k: (t,x')\in E\}}|g(t)|^p dm_k(t)dm_{d-k}(x').
$$
Insofar as the space $L^p(\R^k)$ is separable, we therefore find that, for $m_{d-k}$-almost every $x'\in \R^{d-k}$,
$$C\int_{\{t\in \R^k: (t,x')\in E\}}|g(t)|^p dm_k(t)\geq 1\text{ for every }g\in \mathcal{F},
$$
 where $\mathcal{F} = \{g\in L^p(\R^k): \|g\|_p=1\text{ and }\spec(g)\in [-\beta, \beta]^k\}$.

Fix any $g\in \mathcal{F}$ (it is clearly a non-empty set). Then there exists $M>0$ such that
$$\int_{\{t\in \R^k:|g(t)|\geq M\}}|g(t)|^pdm_k(t)\leq \frac{1}{4C},$$
along with $\ell>0$ such that $$\int_{\R^k\backslash [-\ell, \ell]^k}|g(t)|^pdm_k(t)\leq \frac{1}{4C}.
$$
Taking into account the fact that for every $z'\in \R^k$, $g(\cdot-z')\in \mathcal{F}$, we therefore get that, for $m_{d-k}$-almost every $x'\in \R^{d-k}$,
$$M m_k(\{t\in z'+[-\ell, \ell]^k: (t,x')\in E\})\geq \frac{1}{2C}\text{ for every }z'\in \R^k.
$$
Therefore (1) holds.

\textbf{(1)$\implies$(2).}  We may assume that $\ell=1$ by replacing\footnote{More precisely, if $E$ satisfies (\ref{coordinatecubes}) for every cube of length atleast $\ell$, then $E_{\ell}=\frac{1}{\ell}E$ satisfies (\ref{coordinatecubes}) for every cube of length $1$.  But now if $f$ satisfies $\supp(\wh{f}\,)\subset [-\beta,\beta]^k\times \R^{d-k}$, then $f_{\ell}:=f(\ell\cdot\,)$ satisfies $\wh{f_{\ell}} = \frac{1}{\ell^d}f(\tfrac{\cdot}{\ell})$ and hence $\supp(\wh{f_{\ell}}\,)\subset [-\beta \ell,\beta \ell]^k\times \R^{d-k}$, but also $\int_{E_{\ell}}|f_{\ell}|^p dm_d = \ell^{-d}\int_E|f|^p dm_d$, $\int_{\R^d}|f_{\ell}|^p dm_d = \ell^{-d}\int_{\R^d}|f|^p dm_d$} $\beta$ by $\beta\ell$.   By modifying the set $E$ by a set of Lebesgue measure zero in $\R^d$ (which does not change the integral in (2)), we may assume that the condition in (1) holds for every $x'\in \R^{d-k}$.

 Suppose $\|f\|_{L^p(\R^d)}=1$.  We also may assume without loss of generality that $f\in \mathcal{S}(\R^d)$.


Choose $\psi_0\in \mathcal{S}(\R^k)$ with $\wh{\psi}\equiv 1$ on $[-1,1]^k$ and $\spec(\psi_0)\subset [-2,2]^k$.  Put $\psi = \psi_0(\cdot/\beta)$.  We first claim that we have $f = f*_k\psi$, where $*_k$ denotes a convolution in the first $k$ variables.  To see this, for $\xi\in \R^k$ and $\eta\in \R^{d-k}$, write
\begin{equation}\begin{split}\nonumber\wh{f}(\xi,\eta) & = \wh{f}(\xi, \eta)\wh{\psi}(\xi) \\& = \int_{\R^{d-k}}\int_{\R^k}\int_{\R^k} f(t,x')\psi(s)e^{-2\pi i \xi\cdot (t+s)}dm_k(t)dm_k(s)e^{-2\pi i x'\cdot \eta}dm_{d-k}(x')\\
&=\int_{\R^{d-k}}\int_{\R^k}\int_{\R^k} f(t)\psi(\tau-t)dm_k(t)e^{-2\pi i \xi\cdot \tau}dm_k(\tau)e^{-2\pi i x'\cdot \eta}dm_{d-k}(x')\\
& = \wh{f*_k\psi}(\xi, \eta).
\end{split}
\end{equation}

Fix $x'\in \R^{d-k}$, and set $F=f(\cdot, x')$ so $F:\R^k \to \R$.
Then for any multi-index $\alpha\in \Z_+^k$, we have
$F = \psi*_k\psi*_k\cdots\psi*_kF$, where there are $|\alpha|:=\sum_{j=1}^k \alpha_j$ convolutions of $\psi$.  Then, with $C_0=\|\nabla \psi_0\|_{L^1(\R^k)}$, we have
\begin{equation}\begin{split}\nonumber\|D^{\alpha}F\|_{L^p(\R^k)} &\leq \beta^{|\alpha|}\|\nabla\psi_0\|_{L^1(\R^k)}^{|\alpha|}\|F\|_{L^p(\R^k)}= (\beta C_0)^{|\alpha|}\|F\|_{L^p(\R^k)}.
\end{split}\end{equation}


Fix $A>C_0$ to be chosen momentarily.  We split $\R^k$ into cubes of length $1$.  We call such a cube $I$ \emph{good} if
$$\int_{I}|D^{\alpha}F|^pdm_k\leq (\beta A)^{|\alpha|p}\int_I |F|^p dm_k\text{ for every }\alpha\in \Z_+^k.
$$
Observe that, insofar as there are at most $(n+1)^k$ $\alpha\in \Z_+^k$ with $|\alpha|=n$
\begin{equation}\begin{split}\nonumber\sum_{I\text{ not good}}\int_I |F|^pdm_k&\leq \sum_{n=1}^{\infty}\sum_{\alpha\in \Z^k_+:|\alpha|=n}\frac{1}{(\beta A)^{|\alpha|p}}\int_{\R^k}|D^{\alpha}F|^pdm_k\\&\leq \sum_{n=1}^{\infty}(n+1)^k\frac{C_0^{np}}{A^{np}}\int_{\R^k}|F|^pdm_k.
\end{split}\end{equation}
Therefore, if $A$ is large enough in terms of $C_0$ (which fixes $A$ in terms of $k$), then
$$\int_{\R^k}|F|^p dm_k\leq 2\sum_{I \text{ good}}\int_I |F|^p dm_k.
$$

Now fix a good cube $I$, and put $Z = \{t\in I: (t,x')\in E\}$.  By assumption, $m_k(Z)\geq \gamma$.  Since $I$ is good,
$$\int_{I}|D^{\alpha}F|^pdm_k\leq (\beta A)^{p|\alpha|}\int_I |F|^p dm_k\text{ for every }\alpha\in \Z_+^k,
$$
and so by appealing to a standard Remez inequality for analytic functions (as in  \cite{kovrijkine01}, see also \cite[Corollary 2.8]{jaye18}, \cite[Proposition 3.7]{lebeau19}),
$$\int_I |F|^p dm_k\leq C(\gamma, \beta A)\int_Z |F|^p dm_k,
$$
where $C(\gamma, \beta A) = (C_1/\gamma)^{C_1 \beta A}$ for $C_1=C_1(k)$.

Summing over good intervals, we infer that for every $x'\in \R^{d-k}$,
$$\int_{\R^k}|f(t,x')|^pdm_k(t)\leq C(\gamma, \beta A)\int_{\{t\in \R^k: (t,x')\in E\}}|f(t,x')|^p dm_k(t).
$$
Finally, integrating over $x'\in \R^{d-k}$ yields that
$$\int_{\R^d}|f|^p dm_d\leq C(\gamma,\beta A)\int_{\R^{d-k}}\int_{\{t\in \R^k:(t,x')\in E\}}|f(t,x')|^pdm_k(t)dm_{d-k}(x').
$$
The right hand side is bounded by $C(\gamma,\beta A)\int_{E}|f|^pdm_d$.  Setting $C_2=AC_1$ shows that we have found the desired form of constant.
\end{proof}

\begin{rem}\label{nbhdrem} It easily follows from the proof of the direction (2)$\implies$(1) of Proposition \ref{PLSprop} that if $E\subset \R^d$ is such that \emph{there exists} $\beta>0$, $C>0$ such that $$\|f\|_{L^p(\R^d)}\leq C\|f\|_{L^p(E)}
\text{ whenever }\beta_{d-k}(\spec(f))\leq \beta,$$
then \emph{any open neighborhood of $E$ satisfies the $k$-GCC}.
 \end{rem}









\section{Proof of the main results}

We will use Theorem~\ref{GCCcharac} from the previous section to derive quantitative uniqueness properties for functions with spectrum contained in more complicated sets. Our main results (Theorems~\ref{sphere} and~\ref{manifold}) are a simple consequence of the following more general result. 

\begin{thm}\label{kdimprop}  
Fix $k\in \{1,\dots, d\}$, $\beta>0$, $\gamma\in (0,1)$, $\ell>0$, and $\delta\in (0,\ell)$.  There exists $R>0$ and $C>0$ such that if
 \begin{enumerate}
 \item $\Gamma\subset \R^d$ satisfies that for any ball $B$ of radius $R$ centred on $\Gamma$, $$\beta_{d-k}(B\cap \Gamma)<\beta,\text{ and}$$
 \item $E$ satisfies the $k$-dimensional $(\ell, \gamma)$-GCC,
  \end{enumerate}
 then for any $f\in L^2(\R^d)$ with $\spec(f)\subset \U_{\beta}(\Gamma)$,
$$\|f\|_{L^2(\R^d)}\leq C\|f\|_{L^2(\U_{\delta}(E))}.
$$
\end{thm}

In contrast with Theorem~\ref{GCCcharac}, observe that in the conclusion of Theorem~\ref{kdimprop}, we only control the $L^2$-norm of $f$ by its norm on a $\delta$-neighborhood of a $k$-GCC set.  In the generality that Theorem~\ref{kdimprop} is stated, one cannot expect $R$ and $C$ to be bounded independently of $\delta$.  This can be seen as a direct consequence of the sharpness of the classical Ingham inequality for non-harmonic trigonometric series \cite{ingham36, jaffard01}.  The assumption that $\delta\in (0,\ell)$ in the statement of Theorem \ref{kdimprop} is not restrictive since $\U_{\delta}(E)\subset \U_{\delta'}(E)$ for $\delta'>\delta$.

Inspecting the proof, one can see that Theorem~\ref{kdimprop} holds with $C$ and $R$ both taking the form
\begin{equation}\label{constantform}C_1\Bigl(\frac{\ell}{\delta}\Bigl)^{d+1} \Bigl(\frac{C_0}{\gamma}\Bigl)^{C_0\ell \beta}
\end{equation}
where $C_0=C_0(k)$ and $C_1=C_1(k,d)$.

\begin{proof}[Proof of Theorem \ref{kdimprop}]  
Fix $k\in \{1,\dots, d\}$, and $\beta\geq 1$. As in the proof of Proposition \ref{PLSprop}, we may suppose that $\ell=1$ (this means replacing $\delta$ by $\delta/\ell\in (0,1)$).

Choose $\xi_{\ell}$ to be a maximal $1/2$-separated subset of $\Gamma$.  Then we may cover $\U_{\beta}(\Gamma)$ by cubes $Q_{\ell} = Q(\xi_{\ell},2)$ of sidelength $2$ centered on $\xi_{\ell}\in \Gamma$. 

Fix $\psi\in C^{\infty}_0(B(0,1))$ with $\int_{\R^d}\psi dm_d=1$ and $|\wh{\psi}(\xi)|\gtrsim 1$ on $Q(0,2)$.  

Set $\varphi = \delta^{-d}\psi\bigl(\delta^{-1}\cdot)\in C^{\infty}_0(B(0, \delta))$ so that $\wh{\varphi} = \wh{\psi}(\delta\,\cdot\,)$.  Then, since $\delta\in (0,1)$, $\int_{\R^d}\varphi dm_d=1$,  $|\wh{\varphi}(\xi)|\gtrsim 1$ on $Q(0,2)$, and for any $m\in \N$,
$$|\wh\varphi(\xi)|\lesssim_{m} \frac{1}{\textcolor{black}{\delta^{m}}(1+|\xi|)^{m}}.
$$

Fix $\ell$ fixed, consider the function
$$g_{\ell}= f*\mathcal{F}^{-1}\wh{\varphi}(\cdot-\xi_{\ell}),
$$
so that $\|\wh{f}\|_{L^2(Q_{\ell})}\lesssim \|\wh{g_{\ell}}\|_{L^2(\R^d)}$.

Now split $g_{\ell}=g_{\ell}^{(1)}+g_{\ell}^{(2)}$ where $g_{\ell}^{(1)} = \mathcal{F}^{-1}[\wh{g_{\ell}}\chi_{\{\xi:|\xi-\xi_{\ell}|\leq R\}}]$.  Observe that, for fixed $m\in \mathbb{N}$,
\begin{equation}\label{g2small}\begin{split}
\|g_{\ell}^{(2)}\|_{L^2(\R^d)}^2 &= \int_{|\xi-\xi_{\ell}|> R}|\wh{f}(\xi)|^2|\wh{\varphi}(\xi-\xi_{\ell})|^2dm_d(\xi) \\&\lesssim_m \textcolor{black}{\delta^{-m}} \int_{|\xi-\xi_{\ell}|>R} |\wh{f}(\xi)|^2\frac{1}{|\xi-\xi_{\ell}|^{m}}dm_d(\xi).
\end{split}\end{equation}


By assumption $\supp(\wh{g_{\ell}^{(1)}}\,)$ is contained in the $\beta$-neighborhood of a $(d-k)$-plane, and so Theorem \ref{GCCcharac}, with $A=\bigl(\frac{C}{\gamma}\bigl)^{C\beta}$,
$$\|g_{\ell}^{(1)}\|_{L^2(\R^d)}\lesssim A \|g_{\ell}^{(1)}\|_{L^2(E)}.
$$
Whence, applying (\ref{g2small}) twice, we infer that for every $m\in \mathbb{N}$,
\begin{equation}\begin{split}\nonumber\|\wh{f}\|^2_{L^2(Q_{\ell})}&\lesssim \|g_{\ell}\|_{L^2(\R^d)}^2\\&\lesssim_m A\|g_{\ell}^{(1)}\|_{L^2(E)}^2 +\textcolor{black}{\delta^{-m}}\int_{|\xi-\xi_{\ell}|>R} |\wh{f}(\xi)|^2\frac{1}{|\xi-\xi_{\ell}|^{m}}dm_d(\xi)\\
&\lesssim_m A \|g_{\ell}\|_{L^2(E)}^2+A\textcolor{black}{\delta^{-m}}\int_{|\xi-\xi_{\ell}|>R} |\wh{f}(\xi)|^2\frac{1}{|\xi-\xi_{\ell}|^{m}}dm_d(\xi).
\end{split}\end{equation}
We next would like to sum this inequality over $\ell$, using the fact that $\U_{\beta}(\Gamma)\subset \bigcup_{\ell}Q_{\ell}$.

Due to the support property of $\varphi$, we have that $g_{\ell} = (f\chi_{\U_{\delta}(E)})*\mathcal{F}^{-1}\wh{\varphi}(\cdot-\xi_{\ell})$ on $E$.  Whence,
\begin{equation}\begin{split}\nonumber \|g_{\ell}\|_{L^2(E)}^2 &\leq \int_{\R^d} |\wh{f\chi_{\U_{\delta}(E)}}(\xi)|^2|\wh{\varphi}(\xi-\xi_{\ell})|^2dm_d(\xi) \\&\lesssim \delta^{-d-1}\int_{\R^d} \frac{|\wh{f\chi_{\U_{\delta}(E)}}(\xi)|^2}{1+|\xi-\xi_{\ell}|^{d+1}}dm_d(\xi),
\end{split}\end{equation}
but,
\begin{equation}\nonumber\sum_{\ell} \int_{\R^d} \frac{|\wh{f\chi_{\U_{\delta}(E)}}(\xi)|^2}{1+|\xi-\xi_{\ell}|^{d+1}}dm_d(\xi) \lesssim \int_{\R^d} |\wh{f\chi_{\U_{\delta}(E)}}|^2dm_d,
\end{equation}
and so
\begin{equation}\label{mainterm}
\sum_{\ell}\|g_{\ell}\|^2_{L^2(E)}\lesssim \delta^{-d-1}\|f\|_{L^2(\U_{\delta}(E))}^2.
\end{equation}
On the other hand, for any $\xi\in \R^d$, and for any $k\geq 1$, there can be at most $C 2^{kd}R^d$ of the $1/2$-separated points $\xi_{\ell}$ in an annulus $A_k(\xi):=B(\xi, 2^{k}R)\backslash B(\xi, 2^{k-1}R)$.  Therefore,
\begin{equation}\begin{split}\sum_{\ell}\frac{\chi_{\{\ell:|\xi_{\ell}-\xi|>R\}}(\ell)}{|\xi-\xi_{\ell}|^m}&\leq \sum_{k=1}^{\infty}\sum_{\ell\,:\, \xi_{\ell}\in A_k(\xi)}\frac{1}{|\xi-\xi_{\ell}|^m}\lesssim \sum_{k=1}^{\infty}\frac{2^{kd}R^d}{2^{(k-1)m}R^m}\\&\lesssim_m \frac{1}{R^{m-d}}\text{ provided that }m>d.
\end{split}\end{equation}
Consequently, if we set $m=d+1$, then
\begin{equation}\label{errorterm}\sum_{\ell} \int_{\{|\xi-\xi_{\ell}|>R\}} \frac{|\wh{f}(\xi)|^2}{|\xi-\xi_{\ell}|^{m}}dm_d(\xi)\lesssim \frac{1}{R}\|f\|_{L^2(\R^d)}^2.
\end{equation}
Combining (\ref{mainterm}) and (\ref{errorterm}) results in
$$\|f\|_{L^2(\R^d)}^2\lesssim \frac{A}{\delta^{d+1}}\|f\|_{L^2(\U_{\delta}(E))}^2+\frac{A}{\delta^{d+1}R}\|f\|_{L^2(\R^d)}^2.
$$
If $R$ is sufficiently large multiple of $A\delta^{-d-1}$, then the second term on the right hand side can be hidden in the left hand side, and we get
$$\|f\|_{L^2(\R^d)}^2\lesssim \frac{A}{\delta^{d+1}}\|f\|_{L^2(\U_{\delta}(E))}^2.$$
The theorem is proved.
\end{proof}

We are now ready to prove our main result.

\begin{proof}[Proof of Theorem~\ref{manifold}] 
Fix $\delta>0$.   Choose $S$ large enough to be able to apply Theorem \ref{kdimprop} with $R$ replaced by $S$.   Since $\Sigma$ is a compact $(d-k)$-dimensional $C^1$-manifold embedded in $\R^d$, we can find a function $\sigma:\R_+\to \R_+$ with $\sigma(r)/r\to 0$ as $r\to 0$, such that if $x\in \Sigma$, then $\beta_{d-k}(\Sigma\cap B(x,r))\leq \sigma(r)$.  Fix $R\gg S$.  Then for any $x\in R\Sigma$,
$$\beta_{d-k}(R\Sigma\cap B(x,S)) = R\beta_{d-k}(\Sigma\cap B(x/R, S/R))\leq R\sigma(S/R)\leq  \beta.
$$
provided that $R$ is large enough.   But then we apply Theorem \ref{kdimprop} to conclude that the desired statement holds for sufficiently large $R$.  On the other hand, for small $R$, we can instead apply Theorem \ref{GCCcharac} by covering $\U_{\beta}(R\Sigma)$ by a strip of width $O_{\Sigma, \beta}(R)$.
\end{proof}

Finally, Theorem~\ref{sphere} is an immediate consequence of Theorem~\ref{manifold} obtained by taking $\Sigma=\mathbb{S}^{d-1}$.  Moreover, in this case $\sigma(r)\lesssim r^2$, and inserting the explicit form of the parameters in Theorem \ref{kdimprop} (see (\ref{constantform})) into the proof above we readily derive an explicit constant in Theorem~\ref{manifold}.

\section{Decay Rates for Damped Wave Semigroups}

We now use Theorem~\ref{sphere} to prove energy decay rates for damped fractional wave equation (Theorem~\ref{thm:wave}).
Throughout this section, for a function $f$ we set $\|f\|: = \|f\|_{L^2(\R^d)}$.  For an operator $A$ between two normed spaces, $\|A\|$ denotes the operator norm.

We view the equation (\ref{eq:1}) as the following semigroup. Setting $W(t) = (w(t),w_t(t))$, 
	\[ \dfrac d{dt}W(t) = \cala_\gamma W(t) \]
where $\cala_\gamma : H^{s} \times H^{s/2} \to \hls$ is densely defined by $A_\gamma(u_1,u_2) = (u_2,-\ds u_1 - \gamma u_2)$.
The Sobolev space $H^r$ for $r \in \R$ is defined by the decay of the Fourier transform:
	\[ H^r := \left\{ u \in L^2 : \|u\|_{H^r}^2 = \int_{\R^d} (|\xi|^2+1)^r |\wh u(\xi)|^2 \,d\xi < \infty \right\}. \]
It is easy to verify that $\cala_0$ is skew adjoint. That $W(t)=e^{tA_\gamma}$ is  a strong semigroup of contractions follows by the standard theory \cite{pazy83} since $A_\gamma$ is closed with dense range, and for any $U = (u_1,u_2) \in H^{s} \times H^{s/2}$,
	\[ \re \lip \cala_\gamma^*U,U\rip_{\hls} = \re \lip \cala_\gamma U,U\rip_{\hls} \]
	\[ = \re \lip \cala_0 U,U \rip_{\hls} - \lip \gamma u_2,u_2 \rip_{L^2} = - \lip \gamma u_2,u_2 \rip_{L^2} \le 0. \]
Notice that $E(t) = \|e^{t\cala_\gamma} (w(0),w_t(0))\|_{\hls}$, so the energy decay rates in Theorem \ref{thm:wave} can be rewritten as
	\[ \|e^{t\cala_\gamma}\cala_\gamma^{-1} \| = O(t^{\tfrac{-s}{4-2s}}) \quad 0 < s< 2, \]
	\[ \|e^{t\cala_\gamma}\| \le Ce^{-\omega t}\quad s \ge 2. \]
Once we can establish
	\begin{equation}\label{eq:gamma-res} \|(\cala_\gamma-i\lambda)^{-1}\| \le C\max\{ (|\lambda|+1)^{\tfrac 4s -2},1\}, \end{equation}
the decay rates will follow from the following two results from semigroup theory.
\begin{thm}[Gearhart-Pr\"{u}ss Test \cite{pruss84}]\label{thm:pruss}
Let $e^{tA}$ be a $C_0$-semigroup in a Hilbert space and assume there exists $M>0$ such that $\|e^{tA}\| \le M$ for all $t \ge 0$. Then, there exists $C,\omega>0$ such that
	\[ \|e^{tA}\| \le Ce^{-\omega t} \]
if and only if $i\R \subset \rho(A)$ and $\sup_{\lambda \in \R} \|(A-i\lambda)^{-1}\| < \infty$.
\end{thm}

For the polynomial decay, we use the following result from \cite[Theorem 2.4]{borichev10}:
\begin{thm}[Borichev-Tomilov \cite{borichev10}]\label{thm:borichev}
Let $e^{tA}$ be a $C_0$-semigroup on a Hilbert space. Assume there exists $M>0$ such that $\|e^{tA}\| \le M$ for all $t \ge 0$ and $i\R \subset \rho(A)$. Then for a fixed $\alpha > 0$, 
	\[ \|e^{tA}A^{-1}\|=O(t^{-1/\alpha}) \mbox{ as } t\to \infty \]
if and only if $\|(A-i\lambda)^{-1}\|=O(\lambda^\alpha)$ as $\lambda \to\infty$.
\end{thm}

The main step in establishing (\ref{eq:gamma-res}) is the following resolvent estimate for the fractional Laplacian from sets which satisfy the Geometric Control Condition.
\begin{prop}\label{prop:res}
Let $E \subset \R^d$ satisfy the GCC and $\delta>0$. Then, there exists $C>0$ such that
	\[ \|f\|^2 \le C(1+\lambda)^{\tfrac 2s -2} \|(\ds-\lambda)f\|^2 + \|f\|_{L^2(\U_{\delta}(E))}^2 \]
for all $f \in L^2(\R^d)$, $\lambda \ge 0$.
\end{prop}
\begin{proof}
First, consider the annulus $A_\mu= \{ \xi \in \R^d : \bigr| |\xi|-\mu \bigr| \le 1 \}$ for $\mu \ge 0$.   Appealing to Corollary \ref{sphere}, we find a constant 
 $C>0$ (independent of $\mu$) such that
	\begin{equation}\label{eq:uncp} \|f\| \le C\|f\|_{L^2(\U_{\delta}(E))}  \mbox{ whenever  }\spec f \subset A_\mu.\end{equation}

Now, define the Fourier restriction $\wh{P_\lambda f} := \chi_{\wt A_\lambda} \wh f$ where
	\[ \wt A_\lambda = \{ \xi \in \R^d :|(|\xi|^2+1)^{1/2}-\lambda^{1/s}| \le 1/2 \}. \]
Then, for $\lambda \ge 3^s$, $\wt A_\lambda \subset A_{\lambda^{1/s}}$ so for any $f \in L^2(\R^d)$,
\begin{equation}\begin{split}\nonumber \|f\|^2 & \le \|(I-P_\lambda)f\|^2 + C\|P_\lambda f\|^2_{L^2(\U_{\delta}(E))} \\&\le (2C+1)\|(I-P_\lambda)f\|^2 + 2C\|f\|^2_{L^2(\U_{\delta}(E))}. \end{split}\end{equation}
Finally, using the fact that $|\tau^s-\lambda| \ge c(1+\lambda)^{1-\tfrac 1s}$ if $|\tau-\lambda^{1/s}| \ge 1$ (see Lemma 1 in \cite{green20decay}), we have
	\begin{equation}\begin{split}\nonumber \|(I-P_\lambda)f\|^2 &= \int_{\wt A_\lambda^c} |\wh f(\xi)|^2 \, d\xi \\&\le C(1+\lambda)^{\tfrac 2s -2} \int_{\R^d} \left[(|\xi|^2+1)^{s/2}-\lambda\right]^2|\wh f(\xi)|^2 \, d\xi \end{split}\end{equation}
which completes the proof.
\end{proof}

To prove (\ref{eq:gamma-res}) one can follow the strategy from \cite{green20decay} which we briefly outline. First, the parallelogram identity and the positivity of $\ds$ yields
	\begin{equation}\label{eq:wave-res} c\|U\|^2_{\hls} \le (1+|\lambda|)^{\tfrac 4s -2}\|(\cala_0-i\lambda)U\|^2_{\hls} + \|u_2\|_{L^2(\U_{\delta}(E))}^2 \end{equation}
from Proposition \ref{prop:res} with $s$ replaced by $s/2$. To get the desired estimate (\ref{eq:gamma-res}), notice first that since $\gamma$ is bounded, for any $\eps >0$
	\[ \int_\ell \gamma(x) dm_1(x) \le \|\gamma\|_\infty m_1( \{\gamma \ge \eps\} \cap \ell ) + \eps L \]
all line segments $\ell$ of length $L$. Thus, if $\gamma$ satisfies (\ref{eq:gamma-gcc}), then for $\eps$ small enough, $\{\gamma \ge \eps \}$ satisfies the GCC. Moreover, since $\gamma$ is uniformly continuous, there exists $\delta>0$ such that $ |\gamma(y)| \ge \eps/2$ for $y$ in a $\delta$-neighborhood of $\{\gamma \ge \eps\}$. Thus, taking $E=\{\gamma \ge \eps\}$,
	\[ \|u_2\|_{L^2(\U_{\delta}(E))} \le 2{\eps}^{-1} \|\gamma u_2\|_{L^2(\R^d)}. \]
Finally, using the triangle inequality and the fact that
	\[ \|\gamma u_2 \|^2 \le C \lip \gamma u,u\rip = -C\re \lip (\cala_\gamma -i\lambda U,U \rip_{\hls} \]
	\[\le \eta^{-1} \|(\cala_\gamma-i\lambda) U\|_\hls^2 + \eta \|U\|_\hls^2 \]
for any $\eta>0$, we have
	\begin{align*}
		c\|U&\|_\hls^2 \le (1+|\lambda|)^{\tfrac 4s -2}\| (\cala_\gamma-i\lambda)U\|^2_\hls + ((1+|\lambda|)^{\tfrac 4s -2}+4\eps^{-2})\|\gamma u_2\|_{L^2}^2 \\
			&\le C\max\{ (1+|\lambda|)^{\tfrac 4s-2},(1+|\lambda|)^{\tfrac 8s -4},1\} \|(\cala_\gamma-i\lambda)U\|_\hls^2 + c/2 \|U\|^2_\hls
	\end{align*}
by choosing $\eta$ appropriately. This gives the desired estimate (\ref{eq:gamma-res}).

The converse is a consequence of only the exponential decay case, so we fix $s=2$. By the Gearhart-Pruss Test, one has the resolvent estimate
	\[ \|U\|_\hls \le C\|(\cala_\gamma-i\lambda)U\|_\hls \]
for all $\lambda \in \R$ and all $U=(u_1,u_2) \in \hls$. Taking $U=(\dssi u,iu)$ for some $u \in L^2(\R^d)$, this implies
	\[ c\|u\|^2 \le \|(\dss-\lambda)u\|^2 + \|\gamma u\|^2. \]

Now fix $\kap>0$ to be small, and set $\wt A_\lambda = \{ \xi \in \R^d : \bigr| (|\xi|^2+1)-\lambda \bigr| \le \kap \}$, if $\supp \wh u \subset \wt A_\lambda$, then
	\[ \|((-\Delta+1)^{1/2}-\lambda)u\| \le \kap \|u\|. \]
Consequently, if $\kap$ is small enough, then there exists $c>0$ such that for every $\lambda>0$, $c\|u\| \le \|\gamma u\|$ whenever $\spec(u)\subset \wt A_\lambda$. Since this inequality does not see modulation, the conclusion also holds with $\wt A_{\lambda}$ replaced by $ \wt A_\lambda+\xi$ for any $\xi\in \R^d$. 

Next, we notice that inside $\wt A_\lambda$, one can fit a rectangle with side lengths $c_0\kap \times c_0\lambda^{1/2} \times \ldots \times c_0 \lambda^{1/2}$ for some constant $c_0>0$. 
Letting $\lambda\to \infty$, we therefore see that $c\|u\|\leq \|\gamma u\|$ whenever $\beta_{d-1}(\spec(u))\leq c_0\kap$. But now, insofar as $\gamma$ is bounded, we find $\eps>0$ such that  $c\|u\|\leq \|\chi_{\{|\gamma|>2\eps\}} u\|$ whenever $\beta_{d-1}(\spec(u))\leq c_0\kap$.  Employing Remark \ref{nbhdrem} (with $\beta=c_0\kap$), we see that the set $\{\gamma > \eps\}$ satisfies the $1$-GCC, from which we conclude that $\gamma$ satisfies (\ref{eq:gamma-gcc}).

\bibliographystyle{abbrv}
\bibliography{/mnt/c/Users/awgre/OneDrive/Documents/new/references/journal-abbrv-short,/mnt/c/Users/awgre/OneDrive/Documents/new/references/references}

\end{document}